\documentclass[12pt]{amsart}
\usepackage{xy}
\usepackage{bbm}
\usepackage{graphicx}
\usepackage{mathrsfs}
\usepackage{hyperref}
\usepackage{spectralsequences}
\usepackage{amsmath, amssymb, amsthm, latexsym}
\usepackage{amssymb,amscd,amsmath}
\usepackage{latexsym}
\usepackage{MnSymbol}
\usepackage{enumitem}
\usepackage[center]{caption}
\usepackage{tikz}
\newcounter{braid}
\newcounter{strands}

\DeclareMathAlphabet{\bsf}{OT1}{cmss}{bx}{n}

\pgfkeyssetvalue{/tikz/braid height}{1cm}
\pgfkeyssetvalue{/tikz/braid width}{1cm}
\pgfkeyssetvalue{/tikz/braid start}{(0,0)}
\pgfkeyssetvalue{/tikz/braid colour}{black}
\pgfkeys{/tikz/strands/.code={\setcounter{strands}{#1}}}

\makeatletter
\def\cross{%
  \@ifnextchar^{\message{Got sup}\cross@sup}{\cross@sub}}

\def\cross@sup^#1_#2{\render@cross{#2}{#1}}

\def\cross@sub_#1{\@ifnextchar^{\cross@@sub{#1}}{\render@cross{#1}{1}}}

\def\cross@@sub#1^#2{\render@cross{#1}{#2}}

\def\render@cross#1#2{
  \def\strand{#1}
  \def\crossing{#2}
  \pgfmathsetmacro{\cross@y}{-\value{braid}*\braid@h}
  \pgfmathtruncatemacro{\nextstrand}{#1+1}
  \foreach \thread in {1,...,\value{strands}}
  {
    \pgfmathsetmacro{\strand@x}{\thread * \braid@w}
    \ifnum\thread=\strand
    \pgfmathsetmacro{\over@x}{\strand * \braid@w + .5*(1 - \crossing) * \braid@w}
    \pgfmathsetmacro{\under@x}{\strand * \braid@w + .5*(1 + \crossing) * \braid@w}
    \draw[braid] \pgfkeysvalueof{/tikz/braid start} +(\under@x pt,\cross@y pt) to[out=-90,in=90] +(\over@x pt,\cross@y pt -\braid@h);
    \draw[braid] \pgfkeysvalueof{/tikz/braid start} +(\over@x pt,\cross@y pt) to[out=-90,in=90] +(\under@x pt,\cross@y pt -\braid@h);
    \else
    \ifnum\thread=\nextstrand
    \else
     \draw[braid] \pgfkeysvalueof{/tikz/braid start} ++(\strand@x pt,\cross@y pt) -- ++(0,-\braid@h);
    \fi
   \fi
  }
  \stepcounter{braid}
}

\tikzset{braid/.style={double=\pgfkeysvalueof{/tikz/braid colour},double distance=1pt,line width=2pt,white}}

\newcommand{\braid}[2][]{%
  \begingroup
  \pgfkeys{/tikz/strands=2}
  \tikzset{#1}
  \pgfkeysgetvalue{/tikz/braid width}{\braid@w}
  \pgfkeysgetvalue{/tikz/braid height}{\braid@h}
  \setcounter{braid}{0}
  \let\sigma=\cross
  #2
  \endgroup
}
\makeatother

\input xypic
\newtheorem{theorem}{Theorem}%[section]
\newtheorem{proposition}[theorem]{Proposition}

\newtheorem{lemma}[theorem]{Lemma}

\def\Z{\mathbb{Z}}

\def\C{\mathbb{C}}

\def\C{\mathbb{C}}

\def\qed{\hfill$\square$\medskip}

\def\Zpk{\mathbb{Z}/p^{k}}
\def\Zpk1{\mathbb{Z}/p^{k-1}}

\newcommand{\rref}[1]{(\ref{#1})}

\newcommand{\beg}[2]{\begin{equation}\label{#1}#2\end{equation}}

\def\sl2{\widetilde{SL_{2}(\Z)}}

\title[$Q_8$ equivariant cohomology]{On the $RO(G)$-graded coefficients of $Q_8$ equivariant cohomology}
\author{Yunze Lu}
\begin{document}
\maketitle

\begin{abstract}
We calculate the $RO(G)$-graded coefficents of $H\underline{\Z}$, the Eilenberg-MacLane spectrum of constant Mackey functor for quaternion group $Q_8$. 
\end{abstract}

\vspace{10mm}

\section{Introduction}\label{sp1}

\vspace{2mm}

Let $G$ be a compact Lie group. In \cite{mm}, Lewis, May and McClure defined Eilenberg-MacLane $RO(G)$-graded cohomology with Mackey functor coefficients. Calculation of the $RO(G)$-graded cohomology groups of a point for a non-trivial finite group $G$ has been a fundamental question for equivariant stable homotopy theory. Stong \cite{stong} did the computation for $G=\Z/p$ with Burnside ring coefficient when $p$ is a prime number. Study of cases $G=\Z/{p^n}$ and $G=(\Z/p)^n$ follows in the works of various authors \cite{hhr}, \cite{hok}, \cite{hok2}, \cite{hk}, \cite{sk} \cite{clm}, \cite{mz}, etc. Recently the coefficients when $G$ is a dihedral group of order $2p$ is computed by the author and Kriz in \cite{ky}, also by Zou in \cite{zou}. Later the multiplicative strucutres is worked out by Liu \cite{liu}. In \cite{sk3} the coefficients of geometric fixed points of equivariant cohomology with constant coefficient are computed for any finite group, but a full computation of $RO(G)$-graded coefficients has not been known for any nonabelian group beyond dihedral groups $D_{2p}$. The purpose of this paper is to give a calculation of the $RO(G)$-graded cohomology coefficients for $G=Q_8$, the quaternion group, with constant Mackey functor $\underline{\Z}$ coefficient system.  

\vspace{3mm}

In equivariant stable homotopy theory, the role of abelian groups in the non-equivariant world is replaced by Mackey functors. For a finite group $G$, a $G$-Mackey functor could be described by Lewis diagram \cite{l}. Let $H, K$ denote subgroups of $G$, the constant Mackey functor $\underline{\Z}$ assigns $\Z$ to every left coset $G/H$. The restriction maps are all the identity maps, and the transfer $G/H\mapsto G/K$ is multiplication by $|H/K|$. 

\vspace{3mm}

Given a $G$-Mackey functor $\underline{M}$, one can construct Eilenberg-MacLane spectrum $H\underline{M}$ with properties analogous to its non-equivariant counterparts. Also, one can suspend an equivariant specturm not only by ordinary spheres, but also by representation spheres $S^V$, which is the one point compactification of real representation $V$. This suggests that the equivariant homology and cohomology should be graded by $RO(G)$, the real representation ring of $G$, and such grading is also necessary to obtain Spanier-Whitehead duality. We refer to \cite{lms} and \cite{mm} for more details of the framework.

\vspace{3mm}

We will fix the group $G$ to be the quaternion group $Q_8$. Let $G=Q_8$ be presented as 
$$\{i,j\,|\,i^4, i^2j^{-2}, ijij^{-1}\}.$$
The quaternion group $G$ has four one-dimensional real representations, given by scalar action of generators $i$ and $j$: 
$$i\mapsto \pm1,\quad j\mapsto \pm 1.$$
We will denote the trivial representation by $1$ and the other three representations by $\alpha,\beta,\gamma$, whose kernels are respectively $\langle i\rangle, \langle j\rangle, \langle ij\rangle.$ 

The group $G$ also has a four-dimensional irreducible real representation, where $G$ acts by left multiplication, and we will denote this representation by $\rho$. This representation is of quaternionic type.  

Hence the representations $1,\alpha,\beta,\gamma,\rho$ form an additive basis for $RO(G)$, and hence the grading is denoted as $*+k\alpha+\ell\beta+m\gamma+n\rho,$ where $*$ represents the $\Z$-grading. 

\vspace{3mm}

The main tools we use for the calculations are the $G$-equivariant cellular structures on representation spheres and the method of isotropy separation. Many of the results are essentially coming from group homology and cohomology of $Q_8$, the strategy itself is still of interest. 

\vspace{2mm} 

First we will do some reductions on the $RO(G)$-grading. The representations $\alpha,\beta,\gamma$ are symmetric up to automoprhisms of $Q_8$, hence without loss of generality we may assume that $k,\ell$ have the same sign. Furthermore, by universal coefficient theorem and Spanier-Whitehead duality, if we flip all the signs in the grading, i.e., changing $*+k\alpha+\ell\beta+m\gamma+n\rho$ to $-*-k\alpha-\ell\beta-m\gamma-n\rho$, there is an isomophism 
$$H\underline{\Z}^{Q_8}_{*+k\alpha+\ell\beta+m\gamma+n\rho}\cong H\underline{\Z}_{Q_8}^{-*-k\alpha-\ell\beta-m\gamma-n\rho}$$
so we may further restrict to $k\geq \ell \geq 0.$ 
Therefore, we reduce to the calculations of $\Z$-graded homology and cohomology of 
$$\Sigma^{k\alpha+\ell\beta+m\gamma+n\rho}H\underline{\Z}$$
for $k\geq \ell\geq 0.$

\vspace{3mm}

Under the assumption above, we still need to treat $m,n\in \Z$. Consider the following extension $$1\rightarrow \mathbb{Z}/4 \rightarrow Q_8 \rightarrow \Z/2 \rightarrow 1,$$ where the $\Z/4$ is the kernel of $\gamma.$ We look at the sign of $n$ and divide into the following two cases: 

\vspace{2mm}

For a real orthogonal representation $V$, let $S(V)$ denote the unit sphere in $V$, and let $S^V$ denote the one-point compactification of $V$. When $n\geq 0$, consider the isotropy separation sequence
\beg{is1}{S(n\rho)_+\rightarrow S^0\rightarrow S^{n\rho}}

When $n<0$, the Spanier-Whitehead dual of \rref{is1} is:
\beg{is2}{S^{n\rho}\rightarrow S^0\rightarrow \Sigma^{n\rho+1}S(-n\rho)_+}

Smash the sequences \rref{is1} and \rref{is2} with $S^{k\alpha+\ell\beta+m\gamma}.$ The sphere $S(n\rho)$ (or $S(-n\rho)$ when $n<0$) is free, hence the coefficients are essentially Borel homology and Borel cohomology. We will use an explicit $Q_8$-equivariant CW structure of $S(n\rho)_+$ to calculate these coefficients.

\vspace{2mm}

The coefficients after smashing $S^0$ with $S^{k\alpha+\ell\beta+m\gamma}$ were computed by Holler and Kriz in \cite{hk} for $H\underline{\Z/2}$. Here we need both the result for $H\underline{\Z}$ and the connecting maps. To deal with $m\in \Z$, the trick is to first take $\Z/4$-fixed point on the chain level, to work with $\Z/2=Q_8/(\Z/4)$-equivariant chain complexes. With suitable models for chain complexes, we write the chain complex of $\Sigma^{k\alpha+\ell\beta}H\Z$ as direct sum of simple chain complexes, and the connecting map only appears in one summand, and the suspensions by $S^{m\gamma}$ become formal operation on chain complexes.

\vspace{3mm}

The structure of the paper is the following: In Section \ref{sp2} we will describe an explicit $Q_8$-equivariant CW structure on unit spheres $S(n\rho)_+$ and calculate the homology of suspensions of $S(n\rho)_+$, using the complex stability of Borel theories. These will turn out to essentially be group homology and cohomology of $Q_8$. 

\vspace{2mm}

In Section \ref{sp3} we will explain the main strategy of taking $\Z/4$-fixed points, and compute the models of $\Z/2$-equivariant chain complexes. The main results require a significant amount of notation and for this reason will be stated in Section \ref{sp4} where we also carry out the computation of the $\Z/2$-equivariant chain complexes obtained as the $\Z/4$-fixed point of the equivariant chain complexes corresponding to sequence \rref{is1} and \rref{is2}. The main results are stated in Theorem \ref{t1} and Theorem \ref{t2}. In the end we complete the computations by stating the corresponding cohomological result in Theorem \ref{t3}.

\vspace{10mm}

\section{Equivariant CW Structures}\label{sp2}

\vspace{2mm}

As mentioned in the introduction, our computation is based on an explicit $Q_8$-equivariant CW structure on $S(n\rho)$ along with the corresponding chain complex structures. Since this is extremely important to the result, we carry out these computations in detail. 

\vspace{2mm}

For the 4-dimensional representaion $\rho$, by regarding $\mathbb{R}^4$ as spanned by basis $1,i,j,ij$, the actions of group elements permute the basis up to a sign. For example, generator $i$ acts by the following matrix:
$$
\left[
\begin{matrix}
   0 & -1 & 0 & 0\\
   1 & 0 & 0 & 0\\
   0 & 0 & 0 &-1\\
   0 & 0 & 1 & 0
  \end{matrix}
\right]
$$
For $n>0$, let $S(n\rho)$ be the unit sphere in the representaion $n\rho$. Observe that $S(\infty \rho)$ is a model for the universal classifying space $EQ_8$. This gives a guide for how to subdivide $S(n\rho)$ to obtain a $Q_8$-equivariant CW structure. 

\vspace{2mm}

By identifying the non-equivariant underlying space of $S(n\rho)$ as a subspace of $\mathbb{R}^{4n}$, we index the Euclidean coordinates as 
$$\{x_1,y_1,z_1,w_1,...,x_n,y_n,z_n,w_n\}.$$ 
We also use $X_r$ as an abbreviation for $(x_r,y_r,z_r,w_r).$ Let $1\leq r\leq n$, consider the folloing cells for $S(n\rho)$:

\vspace{2mm}
\leftline{\textbf{Type A.}}

Cells $a_{r,1}$ generated by 
$$\{(X_1,...,X_{r-1},x_r,0,0,0,0,...,0\}\in S(n\rho)\,|\,x_r\in [0,1]\}.$$
The cell $a_{r,1}$ has dimension $4r-4$. It is $Q_8$-free.

\vspace{2mm}
\leftline{\textbf{Type B.}}

Cells $b_{r,1},b_{r,2},b_{r,3}$ generated respectively by 
$$\{(X_1,...,X_{r-1},x_r,y_r,0,0,0,...,0\}\in S(n\rho)\,|\,x_r,y_r\in [0,1]\},$$
$$\{(X_1,...,X_{r-1},x_r,0,z_r,0,0,...,0\}\in S(n\rho)\,|\,x_r,z_r\in [0,1]\},$$
$$\{(X_1,...,X_{r-1},x_r,0,0,w_r,0,...,0\}\in S(n\rho)\,|\,x_r,w_r\in [0,1]\}.$$
The cells $b_{r,1},b_{r,2},b_{r,3}$ have dimension $4r-3$. They are $Q_8$-free.

\vspace{2mm}
\leftline{\textbf{Type C.}}

Cells $c_{r,1},c_{r,2},c_{r,3},c_{r,4}$ generated respectively by 
$$\{(X_1,...,X_{r-1},x_r,y_r,z_r,0,0,...,0\}\in S(n\rho)\,|\,x_r,y_r,z_r \in [0,1]\},$$
$$\{(X_1,...,X_{r-1},x_r,y_r,0,w_r,0,...,0\}\in S(n\rho)\,|\,x_r,y_r,w_r\in [0,1]\},$$
$$\{(X_1,...,X_{r-1},x_r,0,z_r,w_r,0,...,0\}\in S(n\rho)\,|\,x_r,z_r,w_r\in [0,1]\},$$
$$\{(X_1,...,X_{r-1},0,y_r,z_r,w_r,0,...,0\}\in S(n\rho)\,|\,y_r,z_r,w_r\in [0,1]\}.$$
The cells $c_{r,1},c_{r,2},c_{r,3},c_{r,4}$ have dimension $4r-2$. They are $Q_8$-free.

\vspace{2mm}
\leftline{\textbf{Type D.}}

Cells $d_{r,1},d_{r,2}$ generated by 
$$\{(X_1,...,X_{r-1},x_r,y_r,z_r,w_r,0,...,0\}\in S(n\rho)\,|\,x_r,y_r,z_r,w_r\in [0,1]\},$$
\begin{align*}
\{(X_1,...,X_{r-1},x_r,y_r,z_r,w_r,0,...,0\}\in S(n\rho)\,\\ |\,x_r,y_r,z_r\in [0,1],w_r\in [-1,0]\}.
\end{align*}
$$$$
The cells $d_{r,1},d_{r,2}$ have dimension $4r-1$. They are $Q_8$-free.

\vspace{2mm}

The spaces given by the generators are homeomorphic to (closed) disks. The attaching maps from the boundary of $n$-cells are equivariant and are mapping to lower dimensional cells. Finally the open cells form a partition of $S(n\rho)$, so it is a regular $G$-CW complex. The topology is quotient topology and agrees with the induced topology on $S(n\rho)$. In other words, these cells give a $Q_8$-equivariant CW decomposition for each $S(n\rho).$ All the cells are free since quaternion numbers form a division algebra.

\vspace{2mm}

As being unit sphere in the representations, the CW structure is regular, i.e., the attaching maps are embeddings, hence the incidence coefficients are either $+1$ or $-1$. By identifying $\mathbb{R}^{4n}$ with $\mathbb{C}^{2n}$, we use the following rule to determine induced orientation of the boundary: the induced orientation followed by the outward normal direction should make up together the standard orientation of $\mathbb{C}^{2n}$. With this rule we derive the following differentials: 

\begin{lemma}\label{dnr} With respect to the CW-structure and orientations described above, the $Q_8$-equivariant cell chain complex of $S(n\rho)$ in the sense of Bredon \cite{bredon} has differentials

\vspace{1mm} 

$da_{1,0}=0$

\vspace{1mm}
\leftline{\rm{For} $1<r\leq n$,}

$da_{r>1,0}=(1+i+j+ij+(-1)+(-i)+(-j)+(-ij))(d_{r,1}-d_{r,2})$

\vspace{2mm}

\leftline{\rm{In the rest, for} $1\leq r\leq n$,}

$db_{r,1}=ia_{r,1}-a_{r,1}$

$db_{r,2}=ja_{r,1}-a_{r,1}$

$db_{r,3}=(ij)a_{r,1}-a_{r,1}$

$dc_{r,1}=b_{r,1}-b_{r,2}-jb_{r,3}$

$dc_{r,2}=b_{r,1}-b_{r,3}+ib_{r,2}$

$dc_{r,3}=b_{r,2}-b_{r,3}-(ij)b_{r,1}$

$dc_{r,4}=-(j)b_{r,3}-(ij)b_{r,1}-(i)b_{r,2}$

$dd_{r,1}=c_{r,1}-c_{r,2}+c_{r,3}-c_{r,4}$

$dd_{r,2} =c_{r,1}-jc_{r,2}+(-ij)c_{r,3}-(-i)c_{r,4}.$

\end{lemma}
\begin{proof} 
This is done by the same method as Lemma 3 in \cite{ky}. As an example we look at the cell $c_{r,1}$ for some $r\geq 1$. The generating cell of $c_{r,1}$ is given by 
\beg{cr1}{\{(X_1,...,X_{r-1},x_r,y_r,z_r,0,0,...,0\}\in S(n\rho)\,|\,x_r,y_r,z_r \in [0,1]\}}
The cell is of dimension $4r-2$ and we consider cells of dimension $4r-3$ to which $c_{r,1}$ attaches, and they are $b_{r,1},b_{r,2}$ and $b_{r,3}$. It remains to determine the incidence numbers between $c_{r,1}$ and these cells. Consider the equivariant cell $b_{r,3}$ generated by  
$$\{(X_1,...,X_{r-1},x_r,0,0,w_r,0,...,0\}\in S(n\rho)\,|\,x_r,w_r\in [0,1]\}.$$
Since $x_r,y_r,z_r \in [0,1]$ in the generator of $c_{r,1}$, it is only attached to the $j$-orbit of $b_{r,3}$. This orbit is given by 
$$\{(jX_1,...,jX_{r-1},0,w_r,x_r,0,0,...,0\}\in S(n\rho)\,|\,x_r,w_r\in [0,1]\}$$
since $j(x_r+w_r(ij))=(w_ri+x_rj)$. 

\vspace{1mm}

We could use the basis 
\beg{o1}{(e_1,ie_1,...,e_{2r-2},ie_{2r-2},e_{2r-1},ie_{2r-1})}
to determine the orientation of the generator of $c_{r,1}$ (identify it by an orientation preserving homeomorphism with the unit disk in $\C^{2r-1}$ since $z_r$ is determined by $X_1,...,X_{r-1},x_r,y_r$).

Similarly the induced orientation of $jb_{r,3}$ as subspace is
\beg{o2}{(e_2,-ie_2,-e_1,ie_1,...,e_{2r-2},-ie_{2r-2},-e_{2r-3},ie_{2r-3},-ie_{2r-1})}
On a point of $jb_{r,3}$ that $c_{r,1}$ attaches, by the rules set above, the induced oritentation is given by 
\beg{o3}{(e_1,ie_1,...,e_{2r-2},ie_{2r-2},ie_{2r-1})}
since juxtaposing with outward normal direction $-e_{2r-1}$ (since in (\ref{cr1}) we have $x_r\geq 0$) gives the same orientation as in (\ref{cr1}). Comparing orientations (\ref{o2}) and (\ref{o3}) gives that the incidence number between $c_{r,1}$ and $jb_{r,3}$ is $-1$, i.e., 
$$dc_{r,1}=...-jb_{r,3}+...$$
Other incidence numbers are computed by the same method.
\end{proof}

With this, we can calculate the $Q_8$-equivariant homology and cohomology of $S(n\rho)$ with coefficient $\underline{\Z}$. As an example we explicitly compute the homology here. Order the cells in the order as listed above ($b_{r,2}$ comes after $b_{r,1}$ for example), by Lemma \ref{dnr}, the chain complex is 
$$\Z^2\xrightarrow{d_3}\Z\rightarrow ... \rightarrow \mathbb{Z}^3\xrightarrow{d_1} \mathbb{Z} \xrightarrow{d_4}\mathbb{Z}^2 \xrightarrow{d_3} \mathbb{Z}^4 \xrightarrow{d_2} \mathbb{Z}^3\xrightarrow{d_1} \mathbb{Z}\rightarrow 0.$$
where the differentials are given by the following matrices:

$$d_1=0, d_2=\left[
 \begin{matrix}
   1 & 1 & -1 & -1\\
   -1 & 1 & 1 & -1\\
   -1 & -1 & -1 & -1
  \end{matrix}
  \right], d_3=\left[
 \begin{matrix}
  1 & 1\\
 -1 & -1\\
1 & 1\\
-1 & -1
  \end{matrix}
  \right], d_4=\left[
\begin{matrix}
8\\
-8
\end{matrix}
\right].$$
Taking homology we get $4$-periodic result
$$H^{Q_8}_{q}(S(n\rho),\underline{\mathbb{Z}}) =\left\{\begin{array}{ll}
\mathbb{Z} & \text{$q=0$}\\
\mathbb{Z}/2\oplus \mathbb{Z}/2 & \text{$0<q<4n-1$, $q\equiv 1$ mod 4}\\
\mathbb{Z}/8 & \text{$0<q<4n-1$, $q\equiv 3$ mod 4}\\
0 & \text{$q>0$ even}
\end{array}\right.$$
When $n\rightarrow \infty$, we recover the group homology of $Q_8$ with coefficient in $\Z$. Note that $Q_8$ is a periodic group of period 4, it acts freely on the $3$-sphere $S^\rho$, and every abelian subgroup of $Q_8$ is cyclic \cite{ce}.

\vspace{2mm}

In fact, we can compute all homology and cohomology of $$\Sigma^{k\alpha+\ell \beta+m\gamma}S(n\rho)_+$$ 
for any suspensions by $\alpha,\beta,\gamma.$ This is because Borel theories are complex stable \cite{sk2}, i.e., if $V$ is a complex representation, then
$$H^*_{Q_8}(\Sigma^{\star+V}S(\infty\rho)_+,\underline{\mathbb{Z}})\cong H^{\star+\text{dim }V}_{Q_8}(\Sigma^\star S(\infty\rho)_+,\underline{\mathbb{Z}}).$$
Since the representations $2\alpha,2\beta, 2\gamma$ are complex, we may reduce $k,\ell,m\in \Z$ to $k,\ell,m=0,1.$ By symmetry of $Q_8$ again, it suffices to calculate suspension by $S^\alpha, S^{\alpha+\beta}, S^{\alpha+\beta+\gamma}$. 

\vspace{2mm}

There are spectral sequences 
$$H_{p}(G,\widetilde{H}_q(X))\Rightarrow \widetilde{H}^{Q_8}_{p+q}(EG_+\wedge X)$$
$$H^{p}(G,\widetilde{H}^q(X))\Rightarrow \widetilde{H}_{Q_8}^{p+q}(EG_+\wedge X)$$
When $X$ is a sphere, the non-equivariant (co)homology of $X$ vanishes in all but one degree, and the spectral sequences collapse. The action of $Q_8$ on the top (co)homology class $\Z$ is either a trivial action which arises when  $X=S^{\alpha+\beta+\gamma}$, or a twisted action which arises when $X=S^{\alpha}$ or $X=S^{\alpha+\beta},$ depending on whether the action preserves orientations. Therefore, the computations reduces to calculate the group (co)homology of $Q_8$ with trivial/twisted coefficients.

\vspace{2mm}

To work this out, one may either compute the group (co)homology, using the universal space $S(\infty \rho)_+$, then cut off at the top differential, or do actual suspension and calculate the chains directly. Using the first method for homology, we get the following chain complex after tensoring with twisted coefficients $\Z$ (which will henceforth be denoted as $\Z^-$) over $Q_8$: 
$$\Z^2\xrightarrow{d'_3}... \rightarrow \mathbb{Z}^3\xrightarrow{d'_1} \mathbb{Z} \xrightarrow{d'_4}\mathbb{Z}^2 \xrightarrow{d'_3} \mathbb{Z}^4 \xrightarrow{d'_2} \mathbb{Z}^3\xrightarrow{d'_1} \mathbb{Z}\rightarrow 0.$$
Suppose that the kernel of the twisted coefficients $\mathbb{Z}^-$ here be $\langle i \rangle$, then the differentials are:
$$d'_1=\left[
\begin{matrix}
0 & -2 & -2
\end{matrix}
\right], d'_2=\left[
 \begin{matrix}
   1 & 1 & 1 & 1\\
   -1 & 1 & 1 & -1\\
   1 & -1 & -1 & 1
  \end{matrix}
  \right], d'_3=\left[
 \begin{matrix}
  1 & 1\\
 -1 & 1\\
1 & -1\\
-1 & -1
  \end{matrix}
  \right], d'_4=\left[
\begin{matrix}
0\\
0
\end{matrix}
\right].$$

Taking homology we get 
$$\widetilde{H}^{Q_8}_{q}(\Sigma^\alpha S(n\rho)_+,\underline{\mathbb{Z}}) =\left\{\begin{array}{ll}
\mathbb{Z}/2 &\text{$0 \leq q \leq 4n-1$, $q\equiv 1,2,3$ mod 4}\\
0 & \text{$0\leq q \leq 4n-1$, $q\equiv 0$ mod 4}
\end{array}\right.$$

Similarly, here are the other results:
$$\widetilde{H}^{Q_8}_{q}(\Sigma^{\alpha+\beta}S(n\rho)_+,\underline{\mathbb{Z}}) =\left\{\begin{array}{ll}
\mathbb{Z}/2 &\text{$0 \leq q \leq 4n+1$, $q\equiv 0,2,3$ mod 4}\\
0 & \text{$0\leq q \leq 4n+1$, $q\equiv 1$ mod 4}
\end{array}\right.$$

$$\widetilde{H}^{Q_8}_{q}(\Sigma^{\alpha+\beta+\gamma}S(n\rho)_+,\underline{\mathbb{Z}}) =\left\{\begin{array}{ll}
\mathbb{Z} & \text{$q=3$}\\
\mathbb{Z}/2\oplus \mathbb{Z}/2 & \text{$0 < q \leq 4n+2$, $q\equiv 1$ mod 4}\\
\mathbb{Z}/8 & \text{$2<q<4n+2$, $q\equiv 3$ mod 4}\\
0 & \text{$q$ even}
\end{array}\right.$$
\vspace{10mm}

\section{The Main Constructions}\label{sp3}

\vspace{2mm}

We keep working under the assumption that $k\geq \ell\geq 0$. The main strategy to deal with suspensions by $S^{m\gamma}$ is to take $\Z/4$-fixed points of the chain complexes so that the $m\gamma$-suspension becomes an operation on the level of $\Z/2$-equivariant chain  complexes. In this section we first do the calculations for the free part $S(n\rho)$ (or $S(-n\rho)$ when $n<0$).

\vspace{2mm}

Let $n\geq 0$. We take the $\mathbb{Z}/4=\langle ij \rangle$-fixed points
$$C_*(S(n\rho)_+,\underline{\Z})^{\Z/4},$$
$$C_*(S(n\rho)_+,\underline{\Z^-})^{\Z/4}.$$
Let the generator of the quotient $C':=Q_8/\langle ij\rangle$ be $\sigma.$ With Lemma \ref{dnr}, it is again routine to calculate the following differentials in the $\mathbb{Z}$ case for $1\leq r\leq n$:

\vspace{2mm}

$db_{r,1}=0$

$db_{r,2}=(\sigma-1)a_{r,1}$

$db_{r,3}=(\sigma-1)a_{r,0}$

$dc_{r,1}=b_{r,1}-b_{r,2}-\sigma b_{r,3}$

$dc_{r,2}=b_{r,1}+b_{r,2}-b_{r,3}$

$dc_{r,3}=-\sigma b_{r,1}+b_{r,2}-b_{r,3}$

$dc_{r,4}=-\sigma b_{r,1}-b_{r,2}-\sigma b_{r,3}$

$dd_{r,1}=c_{r,1}-c_{r,2}+c_{r,3}-c_{r,4}$

$dd_{r,2}=c_{r,1}-\sigma c_{r,2} + \sigma c_{r,3} - c_{r,4}$

$da_{r+1,1}=(4+4\sigma) (d_{r,1}-d_{r,2})$

\vspace{2mm}

Define $C(r)$ for $1\leq r \leq n-1$ to be the following chain complex:
$$0\rightarrow \Z[C'] \xrightarrow{1+\sigma} \Z[C'] \xrightarrow{1-\sigma} \Z[C'] \xrightarrow{4+4\sigma} \Z[C'] \xrightarrow{1-\sigma} \Z[C'] \xrightarrow{1+\sigma} \Z[C']\rightarrow 0,$$
where the free $\Z[C']$-modules are generated by 
$$c_{r+1,2},b_{r+1,2},-a_{r,1},d_{r,2},c_{r,2}-c_{r,3},b_{r,2}-b_{r,3}.$$
Define
$$C(n): \quad 0\rightarrow [d_{n,2}] \xrightarrow{1-\sigma} [c_{n,2}-c_{n,3}]\ \xrightarrow{1+\sigma} [b_{n,2}-b_{n,3}]\rightarrow 0,$$
and 
$$C(0): \quad 0\rightarrow [c_{1,2}] \xrightarrow{1+\sigma} [b_{1,2}]\xrightarrow{1-\sigma} [-a_{1,1}] \rightarrow 0).$$
(Note that all the bottom $\mathbb{Z}[C']$'s are at degree 0). 

\vspace{2mm}

These complexes are connected by chain maps $f_r$ (in fact, they are also differentials in the chain complex) for $1\leq r \leq n-1$:
$$f_r: C(r)[-5] \rightarrow C(r+1), \quad [c_{r+1,2}] \xrightarrow{1-\sigma} [b_{r+1,2}-b_{r+1,3}],$$
and the chain map 
$$f_0: C(0)[-2] \rightarrow C(1), \quad [c_{1,2}]\xrightarrow{1-\sigma} [b_{1,2}-b_{1,3}].$$

\vspace{2mm}

If we quotient out, for each $1\leq r\leq n$, two acyclic complexes
$$0\rightarrow c_{r,1}\rightarrow b_{r,1}-b_{r,2}-\sigma b_{r,3}\rightarrow 0$$
$$0\rightarrow d_{r,1}\rightarrow c_{r,1}-c_{r,2}+c_{r,3}-c_{r,4} \rightarrow 0$$
in $C_*(S(n\rho)_+,\underline{\Z})^{\Z/4}$ and then take the cokernel, the result could be written a totalization of the following double complex:
$$\Theta^+_{n,0}:=\text{Tot}(C(0)\xrightarrow{f_0[2]} C(1)[2] \xrightarrow{f_1[7]} C(2)[7] \rightarrow ...\rightarrow C(n)[5n-3]).$$
With $\underline{\Z^-}$ coefficient, similar calculations define 
$$\Theta^-_{n,0}:=\text{Tot}(C^-(0)\xrightarrow{f^-_0[2]} C^-(1)[2] \xrightarrow{f^-_1[7]} C^-(2)[7] \rightarrow ...\rightarrow C^-(n)[5n-3]),$$
where the $C^-$ chains and $f^-$ chain maps are differed from their counterparts by chaging signs of $\sigma$ in all differentials (note that the chain maps $f_r$'s were noted as differentials). As an example, 
$$C^-(0): 0\rightarrow [c_{1,2}] \xrightarrow{1-\sigma} [b_{1,2}]\xrightarrow{1+\sigma} [-a_{1,1}] \rightarrow 0.$$
For example, when $n=2$, the chain complex $\Theta^+_{2,0}$ could be visualized as

$$
\xymatrix{
\circ \ar[r]^{1-\sigma} &\circ \ar[r]^{1+\sigma} &\circ &  &  & & &\\
 &\circ \ar[r]^{1+\sigma}\ar[ur]^{1-\sigma} &\circ\ar[r]^{1-\sigma} &\circ \ar[r]^{4+4\sigma} &\circ \ar[r]^{1-\sigma} &\circ \ar[r]^{1+\sigma} &\circ &\\
& & & & &\circ \ar[r]^{1+\sigma}\ar[ur]^{1-\sigma} &\circ \ar[r]^{1-\sigma} &\circ
}$$
where each circle represents a $\Z[C'].$

\vspace{2mm}

To have full computations of the $RO(G)$-graded coefficients, we will need to smash the sequence \rref{is1} 
$$S(n\rho)_+\rightarrow S^0\rightarrow S^{n\rho}$$
with $\Sigma^{k\alpha+\ell\beta+m\gamma}H\underline{\Z}$ for $k,\ell,m\in \Z$. Recall that we have assumed $k\geq \ell \geq 0$, the map 
$$S(n\rho)_+\rightarrow S^0$$ 
induces map on the level of $\mathbb{Z}/4$-fixed points of chains. When $m=0$, it is
\beg{cnm}{C_*(\Sigma^{k\alpha+\ell\beta}S(n\rho)_+,\underline{\mathbb{Z}})^{\Z/4}\rightarrow C_*(S^{k\alpha+\ell\beta};\underline{\Z})^{\Z/2},}
since the center acts trivially on the target. 

\vspace{2mm}

Since we are working with $C'\cong \mathbb{Z}/2$-equivariant chain complexes, let $\gamma'$ be the sign representation of $C'$, it is useful to also introduce the following notations for $m\geq 0$:
$$A_s^+=C_*(S^{s\gamma'}),$$
$$A_s^-=C_* (S^{s\gamma'}/S^{\gamma'})[-1].$$ These chain complexes are given explictly by	
$$A_s^+: \mathbb{Z}[C']\xrightarrow{1+(-1)^{s-1}\sigma} \mathbb{Z}[C'] \rightarrow ...\rightarrow \mathbb{Z}[C'] \xrightarrow{1-\sigma} \mathbb{Z}[C'] \xrightarrow{\text{aug}}\mathbb{Z}$$
$$A_s^-: \mathbb{Z}[C']\xrightarrow{1+(-1)^{s}\sigma} \mathbb{Z}[C'] \rightarrow ...\rightarrow \mathbb{Z}[C'] \xrightarrow{1+\sigma} \mathbb{Z}[C'] \xrightarrow{\text{aug}}\mathbb{Z}^-.$$
We have the following decomposition:

\begin{lemma}If $k\geq \ell\geq 0$, then 
$$C_*(S^{k\alpha+\ell\beta};\underline{\Z})^{\Z/2}=\bigoplus_{s=\ell}^{k}A_\ell^{(-1)^s}[s] \oplus \bigoplus_{s=0}^{\ell-1}(A_s^{(-1)^s}[s]\oplus A_s^{(-1)^{s+1}}[s+1]).$$
\end{lemma}

\begin{proof}
Smashing $S^{k\alpha}, S^{\ell\beta}$ together, we have the standard CW structure ofr $S^{k\alpha+\ell\beta}$. Choose a generator of the top cohomology class, and map it by differentials of the chain complex, until it hits a cell which is not free (coming from $S^{\ell\beta}$ given $k\leq \ell$), then take the cokernel of this subcomplex, which turns out to be a direct sum. Then the result follows by induction.
\end{proof}

As an illustration, when $k=7, \ell=5$, the $(\text{ker} \gamma)$-fixed point is decomposed as the direct sum of the blobs in Figure \ref{fig1}. A square represens a copy of $\Z$.
\begin{figure}\label{fig1}
\centering
\includegraphics[width=8cm]{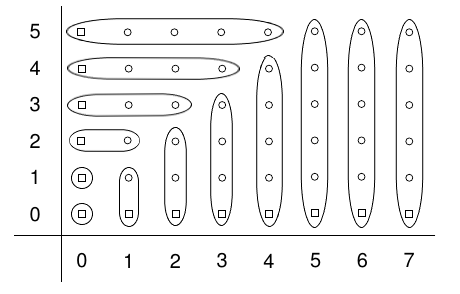}
\caption{$C_*(S^{k\alpha+\ell\beta};\underline{\Z})^{\Z/2}$ when $k=7 ,\ell=5$ }\label{fig1}
\end{figure}

\vspace{2mm}

\vspace{10mm}

\section{The Main Results}\label{sp4}

\vspace{2mm}

To get the result, we compute the cofiber of the map \rref{cnm}, then smash the chain complex with the chain complex of $S^{m\gamma}$, and finally take (co)homology. For this purpose, let $A^+_{s}(m),A^-_{s}(m)$ respectively be the result of smashing $A^+_s,A^-_s$ with $S^{m\gamma}.$ Then we have 

\beg{asm}{
A^\pm_{s}(m)=
\left\{\begin{array}{ll}
A^\pm_{s+m} & \text{$s+m\geq 0$}\\
(A^\pm_{-m-s})^\vee  & \text{$s+m<0$}
\end{array}\right.
}

The homology of the $\mathbb{Z}/2$-fixed points of these chain complexes are given in the following proposition.

\begin{proposition} \label{p1} Taking homology of the $\mathbb{Z}/2$-fixed points, we have 
$$H_{q}(A^+_s)^{\Z/2}=
\left\{\begin{array}{ll}
\mathbb{Z} & \text{$q=s$, $s$ even}\\
\mathbb{Z}/2 & \text{$0\leq q <s$, $q$ even}\\
0 & \text{else}
\end{array}\right.$$

$$H_{q}((A^+_s)^\vee)^{\Z/2}=
\left\{\begin{array}{ll}
\mathbb{Z} & \text{$q=-s$, $s$ even}\\
\mathbb{Z}/2 & \text{$-s\leq q \leq 0$, $q$ odd}\\
0 & \text{else}
\end{array}\right.$$

$$H_{q}(A^-_s)^{\Z/2}=
\left\{\begin{array}{ll}
\mathbb{Z} & \text{$q=s$, $s$ odd}\\
\mathbb{Z}/2 & \text{$0\leq q <s$, $q$ odd}\\
0 & \text{else}
\end{array}\right.$$

$$H_{q}((A^-_s)^\vee)^{\Z/2}=
\left\{\begin{array}{ll}
\mathbb{Z} & \text{$q=-s$, $s$ odd}\\
\mathbb{Z}/2 & \text{$-s\leq q \leq -2$, $q$ even}\\
0 & \text{else}
\end{array}\right.$$

\end{proposition}
\begin{proof} See \cite{l} and \cite{stong}.
\end{proof}

\vspace{2mm}

Now let $\Theta^+_{n,m}$ and $\Theta^-_{n,m}$ respectively be the result of smashing $\Theta^+_{n,0},\Theta^-_{n,0}$ with $S^{m\gamma}$. When $m\leq 2$, $\Theta^+_{n,m}$ is the totalization of the following double chain complex:
$$C(0)_m\xrightarrow{f_0[2-m]} C(1)[2-m] \xrightarrow{f_1[7-m]} C(2)[7-m] \rightarrow ...\rightarrow C(n)[5n-m-3],$$
where $C(0)_m$ is 
$$0\rightarrow \Z[C']\xrightarrow{1+(-1)^0\sigma}\Z[C']\xrightarrow{1+(-1)^1\sigma}\Z[C']\rightarrow ...\xrightarrow{1+(-1)^{1-m}\alpha}\Z[C']\rightarrow 0.$$

\vspace{2mm}

When $m>2$, $\Theta^+_{n,m}$ is the totalization of the following double chain complex: 
$$C(0)_m[2-m]\xrightarrow{\widetilde{f_0}[2-m]} C(1)[2-m] \xrightarrow{f_1[7-m]}  ...\rightarrow C(n)[5n-m-3],$$
where $C(0)_m$ is 
$$0 \rightarrow \Z[C'] \xrightarrow{1+(-1)^{m-3}\sigma}\Z[C']\xrightarrow{1+(-1)^{m-4}\sigma}...\xrightarrow{1+(-1)^0\sigma}\Z[C']\rightarrow 0$$
and $\widetilde{f_0}:C(0)_m\rightarrow C(1)$ is given by $1-\sigma$ at the bottom degree of the both chain complexes.

\vspace{2mm}

As an example, the chain complex $\Theta^+_{2,-2}$ could be presented as 
\begin{equation}{\resizebox{0.9\hsize}{!}{\xymatrix{
\circ \ar[r]^{1-\sigma} &\circ \ar[r]^{1+\sigma} &\circ & & & & & & &\\
 &\circ \ar[r]^{1+\sigma}\ar[ur]^{1-\sigma} &\circ\ar[r]^{1-\sigma} &\circ \ar[r]^{4+4\sigma} &\circ \ar[r]^{1-\sigma} &\circ \ar[r]^{1+\sigma} &\circ & & &\\
& & & & &\circ \ar[r]^{1+\sigma}\ar[ur]^{1-\sigma} &\circ \ar[r]^{1-\sigma} &\circ \ar[r]^{1+\sigma} &\circ \ar[r] &\square
}}}\end{equation}

\vspace{3mm}
 
Similarly, when $m\leq 2$, $\Theta^-_{n,m}$ is the totalization of the following double chain complex:
$$C^-(0)_m\xrightarrow{f^-_0[2-m]} C^-(1)[2-m] \xrightarrow{f^-_1[7-m]} C(2)[7-m] \rightarrow ...\rightarrow C^-(n)[5n-m-3],$$
where $C^-(0)_m$ is 
$$0\rightarrow \Z[C']\xrightarrow{1-(-1)^0\sigma}\Z[C']\xrightarrow{1-(-1)^1\sigma}\Z[C']\rightarrow ...\xrightarrow{1-(-1)^{1-m}\alpha}\Z[C']\rightarrow 0.$$
\vspace{2mm}
When $m>2$, $\Theta^-_{n,m}$ is the totalization of the following double chain complex: 
$$C^-(0)_m[2-m]\xrightarrow{\widetilde{f^-_0}[2-m]} C^-(1)[2-m] \xrightarrow{f^-_1[7-m]}  ...\rightarrow C^-(n)[5n-m-3],$$
where $C^-(0)_m$ is 
$$0 \rightarrow \Z[C'] \xrightarrow{1-(-1)^{m-3}\sigma}\Z[C']\xrightarrow{1-(-1)^{m-4}\sigma}...\xrightarrow{1-(-1)^0\sigma}\Z[C']\rightarrow 0$$
and $\widetilde{f^-_0}:C^-(0)_m\rightarrow C^-(1)$ is given by $1+\sigma$ at the bottom degree of the both chain complexes.

\vspace{2mm}

As an example, the chain complex $\Theta^-_{2,5}$ could be presented as 

$$
\xymatrix{
\circ \ar[r]^{1-\sigma} &\circ \ar[r]^{1+\sigma} &\circ &  &  & &\\
 &\circ \ar[r]^{1+\sigma}\ar[ur]^{1-\sigma} &\circ\ar[r]^{1-\sigma} &\circ \ar[r]^{4+4\sigma} &\circ \ar[r]^{1-\sigma} &\circ \ar[r]^{1+\sigma} &\circ\\
& &\square \ar[r]^{1+\sigma} &\circ\ar[r]^{1-\sigma} &\circ\ar[r]^{1+\sigma} &\circ \ar[ur]^{1-\sigma} &
}$$

\begin{proposition} \label{p2} 
When $m\leq 0$, we have
$$H_{q}((\Theta^+_{n,m})^{\Z/2}) =
\left\{\begin{array}{ll}
\mathbb{Z} & \text{$q=0$ and $m$ even}\\
\mathbb{Z}/2 & \text{$0\leq q \leq -m$, $q\equiv m+1$ mod 2}\\
\mathbb{Z}/2\oplus \mathbb{Z}/2 & \text{$-m\leq q \leq 4n-m-1$, $q\equiv -m+1$ mod 4}\\
\mathbb{Z}/8 & \text{$-m\leq q \leq 4n-m-1$, $q\equiv -m+3$ mod 4}\\
0 & \text{else}
\end{array}\right.$$

When $m>0$, we have 
$$H_{q}((\Theta^+_{n,m})^{\Z/2}) =
\left\{\begin{array}{ll}
\mathbb{Z} & \text{$q=0$ and $m$ even}\\
\mathbb{Z}/2 & \text{$q=m-2$, or}\\
 & \text{$0\leq q \leq 3-m$ and $q\equiv -m+3$ mod 2}\\
\mathbb{Z}/2\oplus \mathbb{Z}/2 & \text{$m+2-4n\leq q \leq m-6$, $q\equiv m-2$ mod 4}\\
\mathbb{Z}/8 & \text{$m-4n\leq q \leq m-4$, $q\equiv m$ mod 4}\\
0 & \text{else}
\end{array}\right.$$

\vspace{2mm}

And when $m\leq 0$, we have
$$H_{q}((\Theta^-_{n,m})^{\Z/2}) =
\left\{\begin{array}{ll}
\mathbb{Z} & \text{$q=0$ and $m$ odd}\\
\mathbb{Z}/2 & \text{$0\leq q \leq -m$, $q\equiv m$ mod 2, or}\\
 & \text{$-m\leq q \leq 4n-m-1$, $q+m\equiv 0,1,2$ mod 4}\\
0 & \text{else}
\end{array}\right.$$

Finally, when $m>0$, we present the homology of $(\Theta^-_{n,m})^{\Z/2})$ as a sum of two chain complexes
$$H_*((\Theta^-_{n,m})^{\Z/2}) \cong H_*((\Theta^-_{n,2})^{\Z/2})[2-m] \oplus H_*(C^-(0)_m[2-m])^{\Z/2}$$
where
$$H_{q}((\Theta^-_{n,2})^{\Z/2}) =
\left\{\begin{array}{ll}
\mathbb{Z}/2 & \text{$-1 \leq q \leq 4n-3$, $q \equiv 0,2,3$ mod 4,}\\
0 & \text{else}
\end{array}\right.$$
and
$$H_{q}(C^-(0)_m[2-m])^{\Z/2}=
\left\{\begin{array}{ll}
\mathbb{Z} & \text{$q=0$ and $m$ odd}\\
\mathbb{Z}/2 & \text{$-1 \leq q \leq 2-m$, $q \equiv m+1$ mod 2,}\\
0 & \text{else}
\end{array}\right.$$

\end{proposition}
\begin{proof}As seen from the above definition, the chain complexes in these have fewer than three copies of $\Z$ in each dimension, also the differentials are simple. Thus we can proceed by direct computation.
\end{proof}

With all the ingredient described, we may write down the first case of the main result.

\begin{theorem}\label{t1} For $k\geq \ell \geq 0$ and $n\geq 0$, as a $\mathbb{Z}/2$-equivariant chain complex,

\begin{align*}
C^{Q_8}_*(S^{k\alpha+\ell\beta+m\gamma+n\rho})^{\Z/4} & = \bigoplus_{s=\ell}^{k-1}A_\ell^{(-1)^s}(m)[s] \oplus \bigoplus_{s=0}^{\ell-1} A_s^{(-1)^s}(m)[s]\\ \nonumber &\oplus A_s^{(-1)^{s+1}}(m)[s+1])\oplus \Theta^{(-1)^{(k+1)(\ell+1)}}_{n,-\ell-m}[\ell].
\end{align*}
The homology of all the chain complexes involved, are computed in Proposition \ref{p1} and Proposition \ref{p2}.
\end{theorem}
\qed

\vspace{2mm}

By Spanier-Whitehead duality, it suffices to furthermore consider the case $n<0$. Essentially, this means $n$ and $k,\ell$ have different signs. If, say, $k,\ell<0$ and $n,m>0$, we can flip all the signs and compute the cohomology instead.

So here we assume $k\geq \ell\geq 0,$ and $n<0$. The cofiber sequence now looks like 
$$S(-n\rho)_+\rightarrow S^0 \rightarrow S^{-n\rho}.$$
Take the dual of this sequence, we have 
$$S^{n\rho}\rightarrow S^0 \rightarrow DS(-n\rho)_+.$$
The dual of $S(-n\rho)_+$ is $\Sigma^{n\rho+1}S(-n\rho)_+$, hence 
$$S^{n\rho}\rightarrow S^0\rightarrow \Sigma^{n\rho+1}S(-n\rho)_+.$$
Define now for $n<0$:
$$\Theta^\pm_{n,m}:=\text{Hom}(\Theta^\pm_{-n,m},\Z),$$
so that $\Theta^+_{n,m}$ and $\Theta^-_{n,m}$ are still results of smashing $\Theta^+_{n,0},\Theta^-_{n,0}$ with $S^{m\gamma}$. Their homology is recorded in the following proposition whose proof is analogous to Proposition \ref{p2}.

\begin{proposition} \label{p3} Let $n<0$.
When $m\leq 0$, we have
$$H_{q}((\Theta^+_{n,m})^{\Z/2}) =
\left\{\begin{array}{ll}
\mathbb{Z} & \text{$q=0$ and $m$ even}\\
\mathbb{Z}/2 & \text{$m-1\leq q \leq -1$, $q\equiv m$ mod 2}\\
\mathbb{Z}/2\oplus \mathbb{Z}/2 & \text{$m+4n \leq q \leq m-1$, $q\equiv m-2$ mod 4}\\
\mathbb{Z}/8 & \text{$m+4n \leq q \leq m-1$, $q\equiv m$ mod 4}\\
0 & \text{else}
\end{array}\right.$$

When $m>0$, we have 
$$H_{q}((\Theta^+_{n,m})^{\Z/2}) =
\left\{\begin{array}{ll}
\mathbb{Z} & \text{$q=0$ and $m$ even}\\
\mathbb{Z}/2 & \text{$q=-m+1$, or}\\
 & \text{$-m+2\leq q \leq -1$ and $q\equiv m$ mod 2}\\
\mathbb{Z}/2\oplus \mathbb{Z}/2 & \text{$5-m\leq q \leq -4n-m-3$, $q\equiv -m+1$ mod 4}\\
\mathbb{Z}/8 & \text{$3-m\leq q \leq -4n-m-1$, $q\equiv -m+3$ mod 4}\\
0 & \text{else}
\end{array}\right.$$

\vspace{2mm}

And when $m\leq 0$, we have
$$H_{q}((\Theta^-_{n,m})^{\Z/2}) =
\left\{\begin{array}{ll}
\mathbb{Z} & \text{$q=0$ and $m$ odd}\\
\mathbb{Z}/2 & \text{$m-1\leq q \leq 1$, $q\equiv m+1$ mod 2, or}\\
 & \text{$m+4n\leq q \leq m-1$, $q-m\equiv 1,2,3$ mod 4}\\
0 & \text{else}
\end{array}\right.$$

Finally, when $m>0$, we present the homology of $(\Theta^-_{n,m})^{\Z/2}$  as a sum of two chain complexes
$$H_*((\Theta^-_{n,m})^{\Z/2}) \cong H_*((\Theta^-_{n,2})^{\Z/2})[2-m] \oplus H_*(C^-(0)_m[2-m])^{\Z/2}$$
where
$$H_{q}((\Theta^-_{n,2})^{\Z/2}) =
\left\{\begin{array}{ll}
\mathbb{Z}/2 & \text{$4n+2 \leq q \leq 0$, $q \equiv 0,1,3$ mod 4,}\\
0 & \text{else}
\end{array}\right.$$
and
$$H_{q}(C^-(0)_m[2-m])^{\Z/2}=
\left\{\begin{array}{ll}
\mathbb{Z} & \text{$q=0$ and $m$ odd}\\
\mathbb{Z}/2 & \text{$m-3 \leq q \leq 0$, $q \equiv m$ mod 2,}\\
0 & \text{else}
\end{array}\right.$$

\end{proposition}
\qed

Suspend by $S^{k\alpha+\ell\beta}$, the connecting map connects at the top degree, and after taking the cofiber, we obtain the next case of our main result:

\begin{theorem}\label{t2} For $k\geq \ell \geq 0$ and $n< 0$, as a $\mathbb{Z}/2$-equivariant chain complex,

\begin{align*}
C^{Q_8}_*(S^{k\alpha+\ell\beta+m\gamma+n\rho})^{\Z/4} & = \bigoplus_{s=\ell}^{k-1}A_\ell^{(-1)^s}(m)[s] \oplus \bigoplus_{s=0}^{\ell-1} A_s^{(-1)^s}(m)[s]\\ \nonumber &\oplus A_s^{(-1)^{s+1}}(m)[s+1])\oplus \Theta^{(-1)^{(k+1)(\ell+1)}}_{n,\ell-m}[\ell].
\end{align*}
The homology of all the chain complexes involved, are computed in Proposition \ref{p1} and Proposition \ref{p3}.
\end{theorem}
\qed

Finally, we complete our discussion by adding the cohomology results. The cohomology of duals of the chain complexes $\Theta$'s and $A$'s are easily derived  from Proposition \ref{p1}, \ref{p2} and \ref{p3} using universal coefficients theorem. What is new is the decomposition of cochain complexes. The answer is the follows:
 
\begin{theorem}\label{t3} For $k\geq \ell \geq 0$ and $n\geq 0$, as a $\mathbb{Z}/2$-equivariant chain complex,

\begin{align*}
C_{Q_8}^*(S^{k\alpha+\ell\beta+m\gamma+n\rho})^{\Z/4} & = \bigoplus_{s=\ell}^{k-1}(A_\ell^{(-1)^s}(m)[s])^\vee \oplus \bigoplus_{s=0}^{\ell-1} (A_s^{(-1)^s}(m)[s])^\vee \\ \nonumber &\oplus (A_s^{(-1)^{s+1}}(m)[s+1])^\vee \oplus \Theta^{(-1)^{(k+1)(\ell+1)}}_{-n,-\ell-m}[\ell].
\end{align*}

For $k\geq \ell \geq 0$ and $n< 0$, as a $\mathbb{Z}/2$-equivariant chain complex,

\begin{align*}
C_{Q_8}^*(S^{k\alpha+\ell\beta+m\gamma+n\rho})^{\Z/4} & = \bigoplus_{s=\ell}^{k-1}(A_\ell^{(-1)^s}(m)[s])^\vee \oplus \bigoplus_{s=0}^{\ell-1} (A_s^{(-1)^s}(m)[s])^\vee\\ \nonumber &\oplus( A_s^{(-1)^{s+1}}(m)[s+1])^\vee\oplus \Theta^{(-1)^{(k+1)(\ell+1)}}_{-n,\ell-m}[\ell].
\end{align*}
\end{theorem}
\qed

\end{document}